\documentclass{amsart}
\usepackage{mathrsfs}
\usepackage{amssymb,latexsym,amsmath,amsthm,enumitem,mathrsfs}
\usepackage{hyperref}
\usepackage{color}
\usepackage{stmaryrd}
\usepackage{comment}
\usepackage{mathrsfs}
\usepackage{lineno}
\usepackage{bbm}
\usepackage{comment}
\usepackage{subcaption}
\usepackage{pgfplots}
\usepackage{graphicx}
\usetikzlibrary{math}

\usepackage[style=alphabetic]{biblatex}
\usepackage{scalerel,stackengine}
\usepackage{tikz}

\usepackage{ esint }
\usepackage{graphicx}
\usepackage{amssymb}
\newtheorem{theorem}{Theorem}
\newtheorem{lemma}[theorem]{Lemma}
\newtheorem{remark}[theorem]{Remark}
\newtheorem{proposition}[theorem]{Proposition}

\newtheorem{definition}[theorem]{Definition}

\newtheorem{corollary}[theorem]{Corollary}
\newtheorem{example}[theorem]{Example}
\newtheorem{conjecture}[theorem]{Conjecture}

\theoremstyle{definition}

\def\set4{\mathcal I}
\def\tup14{(1,2,3,4)}

\newtheorem*{comm*}{Comment}
\newtheorem*{lemma*}{Lemma}

\newcommand{\R}{\mathbb{R}}

\begin{document}

 \author{Paige Bright}
 \address{Department of Mathematics\\
 Massachusetts Institute of Technology\\
 Cambridge, MA 02142-4307, USA}
 \email{paigeb@mit.edu}

 \author{Caleb Marshall}
\address{Department of Mathematics \\ University of British Columbia, 1984 Mathematics Road \\ Vancouver, BC, Canada V6T 1Z2}
\email{cmarshall@math.ubc.ca}

\keywords{}
\subjclass[2020]{28A75, 28A78}

\date{}

\title[A Continuum Erd\H{o}s-Beck Theorem]{A Continuum Erd\H{o}s-Beck Theorem}

\maketitle

\begin{abstract}
  We prove a version of the Erd\H{o}s--Beck Theorem from discrete geometry for fractal sets in all dimensions. More precisely, let $X\subset \R^n$ Borel and $k \in [0, n-1]$ be an integer. Let $\dim (X \setminus H) = \dim X$ for every $k$-dimensional hyperplane $H \in \mathcal{A}(n,k)$, and let $\mathcal L(X)$ be the set of lines that contain at least two distinct points of $X$. Then, a recent result of Ren shows
  $$
\dim \mathcal{L}(X) \geq \min \{2 \dim X, 2k\}.
  $$
  If we instead have that $X$ is not a subset of any $k$-plane, and
  $$
  0<\inf_{H \in \mathcal{A}(n,k)} \dim (X \setminus H) = t < \dim X,
  $$
  we instead obtain the bound
  $$
\dim \mathcal{L}(X) \geq \dim X + t.
  $$
We then strengthen this lower bound by introducing the notion of the ``trapping number'' of a set, $T(X)$, and obtain
\[
\dim \mathcal L(X) \geq \max\{\dim X + t, \min\{2\dim X, 2(T(X)-1)\}\},
\]
as consequence of our main result and of Ren's result in $\mathbb{R}^n$. Finally, we introduce a conjectured equality for the dimension of the line set $\mathcal{L}(X)$, which would in particular imply our results if proven to be true.
\end{abstract}

\tableofcontents

\section{Introduction}

In discrete geometry, a classic result of Beck \cite{beck1983lattice} shows that given $N$ points in $\R^2$, either $\sim N$ of the points lie on a line, or there are $\sim N^2$ lines spanned by the set, i.e. $\sim N^2$ lines which contain at least two of the points. 

\begin{remark}
    \rm{Here, $\sim$ denotes equality up to some positive multiplicative constant. Similarly, $A\lesssim B$ denotes $A\leq CB$ for some multiplicative constant $C$.}
\end{remark}

About two years ago, this result was generalized to the continuum setting by Orponen--Shmerkin--Wang (OSW, \cite{orponen2022kaufman}), showing that given $X\subset \R^2$ Borel, either 
\begin{enumerate}
    \item[\textbf{1)}] there exists a line $\ell$ such that $\dim (X\setminus \ell) < \dim X$, or 
    \item[\textbf{2)}] $\dim \mathcal L(X) \geq \min \{2\dim X, 2\}$, where $\mathcal L(X)$ is the set of all lines containing at least two points of $X$ contained in the metric measure space $\mathcal A(2,1)$, i.e. the space of all lines in $\R^2$.
\end{enumerate}

\begin{remark}\rm{
We define the dimension of a subset of the affine Grassmannian subordinate to suitable metrics on $\mathcal{A}(n,1)$.
For the sake of exposition, we postpone any formal definition until Section \ref{ss:planarprelims}.}
\end{remark}

The proof of the continuum version of Beck's theorem from Orponen, Shmerkin, and Wang used one of their main results on radial projections proven in the plane, as well as a classical lowerbound on the dimension of a (dual) Furstenberg set (see Section \ref{sec:s1.1}). Recently, Kevin Ren proved a higher dimensional radial projection theorem strengthening OSW's result, and as a consequence, Ren was able to derive an analogous continuum version of Beck's theorem in $\R^n$ (\cite{ren2023discretized}, Theorem 1.1 and Corollary 1.5).

Going back to discrete geometry: there is a more general version of Beck's theorem, called the Erd\H{o}s--Beck theorem. This result shows the following: given $N$ points in the plane, $P$, and a parameter $0\leq t \leq N$, suppose that $$\sup_{L\in \mathcal A(2,1)}|P \cap L| \geq N-t.$$ Then $P$ spans $\gtrsim Nt$ many lines.
This result was proven by Jozsef Beck in \cite{beck1983lattice} and resolved a conjecture of Erd\H{o}s--hence, this result is often referred to in the literature as the Erd\H{o}s-Beck Theorem.

Our main focus of this paper is finding a continuum analogue of the Erd\H{o}s--Beck theorem in $\R^n$. In particular, our main result is the following. In what follows, $\mathcal A(n,k)$ is the affine Grassmannian of $k$ planes in $\R^n$.

\begin{theorem}\label{thm:fullbeck}
    Let $X\subset \R^n$ be Borel, and let $k\in [1,n-1]$ be an integer. Then, 
    \begin{enumerate}
        \item[\textbf{1)}] if $\dim (X\setminus H) = \dim X$ for every $H \in \mathcal A(n,k)$, then $\dim \mathcal L(X) \geq \min\{2\dim X, 2k\}$.
        \item[\textbf{2)}] if there exists an $H \in \mathcal A(n,k)$ such that $\dim (X\setminus H) < \dim X$, we let $0 < t \leq \dim X$ such that $\dim (X\setminus P) \geq t$ for all $P \in \mathcal A(n,k)$. Then,
        \[
        \dim \mathcal L(X) \geq \dim X + t.
        \]
    \end{enumerate}
\end{theorem}

Firstly, note that if we are in the second case of the above theorem, and the only value of $t$ such that $\dim X\setminus P \geq t$ for all $P\in \mathcal A(n,k)$ is $t=0$, then one can apply this same theorem with $n=k$ and $X' = X\cap H$. Furthermore, note that, if $\dim X > k$ for some $k$, then Condition $(1)$ of Theorem \ref{thm:fullbeck} necessarily holds. 
Hence, we always have the bound
\begin{equation}\label{eq:trivialbeck}
\dim \mathcal{L} (X) \geq \max\{ 2 (\lfloor \dim X \rfloor - 1), 0 \} \quad \forall X \subset \mathbb{R}^n.
\end{equation}
However, for sets satisfying Condition $(2)$, we may improve this bound by examining the interaction between the sts $X_1 : = X \cap H$ and $X_2 : = X \setminus H$, where $H \in \mathcal{A}(n,k)$ satisfies $\dim (X \setminus H) < \dim X$. We will also further strengthen the lower bound of Theorem \ref{thm:fullbeck} by defining the notion of a ``trapping number'' of a set--a concept related to the estimate \eqref{eq:trivialbeck} appearing above.

\begin{remark}
    \rm{We view our result as a natural continuum analogoue of the Erd\H{o}s-Beck Theorem, with cardinality replaced by Hausdorff dimension. In particular, our estimate gives a lower bound on the dimension of the line set $\mathcal{L}(X)$ as a function of the dimensions of $X$ and its largest dimensional co-planar subset.}
\end{remark} 

Initially, the reader may be surprised that the Erd\H{o}s-Beck bound of $Nt$ becomes a sum bound of $\dim X + t$ in our result. However, such ``logarithmic'' phenomenon occur frequently in these kind of discrete-to-continuum Theorems. Some examples in this vein include partial results towards an Erd\H{o}s-Beck Theorem in the Continuum as in \cite{orponen2022kaufman} and \cite{ren2023discretized}, as well as further afield analogues such as the Kakeya problem in finite fields, where the (conjectured) dimensional estimate of $\dim K = n$ in $\mathbb{R}^n$ becomes a cardinality estimate of $\# K \gtrsim |\mathbb{F}_q|^n$ in the finite field $\mathbb{F}_q$. See, for example, \cite{dvir2009kakeya}.

Let us now discuss how Theorem \ref{thm:fullbeck} can be strengthened by considering the following definition.

\begin{definition}
    Let $X\subset \R^n$ be Borel with $\dim (X) > 0$, and define the \textbf{trapping number} of $X$ to be 
    \[
    T(X) := \begin{cases}\min \{k \in \mathbb{N}_{+} : \text{there exists an } H\in \mathcal A(n,k) \text{ such that } \dim (X\setminus H) < \dim X\} \\
    1 ~~\text{if no such $k$ exists}
    \end{cases}.
    \]
\end{definition}

One can intuitively think of the trapping number of a set $X$ as a way to recognize if (a large proportion of) $X$ is contained in some lower dimensional $k$-plane. Notice then that Ren's continuum Beck's theorem \cite{ren2023discretized} proves the following:
\begin{theorem}[\cite{ren2023discretized}, Corollary 1.5]\label{thm:rensBeck}
    Let $X\subset \R^n$ be Borel. Then,
    \[
    \dim \mathcal L(X) \geq \min\{2\dim X, 2(T(X)-1)\}.
    \]
\end{theorem}

\begin{remark}
    \rm{This was technically proven when $T(X) \geq 2$, though the result is trivial when $T(X) = 1$, so we include it for completeness.}
\end{remark}

Comparing this to our main result, Theorem \ref{thm:fullbeck}, we obtain the following: 

\begin{corollary}\label{cor:fullbeck+Ren}
    Let $X\subset \R^n$ be Borel and let $k\in [0,n-1]$ be an integer. Then, 
    \begin{enumerate}
        \item[\textbf{1)}] if $\dim (X\setminus H) = \dim X$ for every $H \in \mathcal A(n,k)$, then $\dim \mathcal L(X) \geq \min\{2\dim X, 2k\}$.
        \item[\textbf{2)}] if there exists an $H \in \mathcal A(n,k)$ such that $\dim (X\setminus H) < \dim X$, we let $0 < t \leq \dim X$ such that $\dim (X\setminus P) \geq t$ for all $P \in \mathcal A(n,k)$. Then,
        \[
        \dim \mathcal L(X) \geq \max\{\dim X + t, \min\{2\dim X,2(T(X)-1)\}\}.
        \]
    \end{enumerate}
\end{corollary}

\subsection{Examples and a Conjecture} \label{sec:trappingnumber}

We illustrate the relationship between the parameters $k$, $t$, $T(X)$, and the dimension of the set $X$ through examples. 

Firstly, notice that, since $\dim (X \setminus H) \leq \dim X$ for all $k$-planes $H$, we always have that
$$
\dim X + \inf_{H \in \mathcal{A}(n,k)} \dim (X \setminus H) \leq 2 \dim X, \quad \forall 1 \leq k \leq n.
$$
Hence, the term $\dim X + t$ appearing in Corollary \ref{cor:fullbeck+Ren} dominates precisely when $X \subset \mathbb{R}^n$ satisfies the inequality
\begin{equation}\label{eq:BMtermdoms}
2 (T(X) - 1) \leq \dim X + \inf_{H \in \mathcal{A}(n,k)} \dim (X \setminus H) \leq 2 \dim X
\end{equation}
where $\dim X$ satisfies $\dim X \leq k$. The following example furnishes a set $X \subset \mathbb{R}^n$ which satisfies the inequalities \eqref{eq:BMtermdoms} (so Corollary \ref{cor:fullbeck+Ren} obtains a lower bound of $\dim X + t$).

\begin{example} \rm{Let $1 \leq k \leq n-1$ be given and suppose that $X \subset \mathbb{R}^n$ Borel can be written as $X = X_1 \cup X_2$, where each $X_j$ is contained in some distinct $k$-plane $H_j \in \mathcal{A}(n,k)$ and satisfies $k - 1 < \dim X_1 < \dim X_2 = \dim X \leq k$. Then, since
$$
\dim (X \setminus H_2) = \dim (X_1) < \dim (X),
$$
we know that $T(X) \leq k$. However, for any $k-1$ plane $H' \in \mathcal{A}(n, k-1)$, we know that
$$
\dim (X_2 \setminus H') = \dim (X_2) = \dim (X).
$$
Hence, $T(X) \geq k$. So, one knows that:
$$
\min \{ 2 \dim X, 2(T(X) - 1) \} \leq 2k - 2.
$$
However, since
$$
\inf_{H \in \mathcal{A}(n,k)} \dim (X \setminus H ) = \dim (X \setminus H_2) = \dim (X_1),
$$
we have, setting $t = \dim X_1$,
$$
\dim X + t > (k - 1) + (k-1) = 2k - 2
$$
Hence, the bound $\dim X + t$ of Corollary \ref{cor:fullbeck+Ren} dominates in this scenario.}
\end{example}

We now provide some intuition for how the parameters $t$ and $T(X)$ come into play for sets with varying geometric structure, and give an example of a set where the term $\min \{2 \dim X, 2 (T(X) - 1) \}$ does, in fact, dominate $\dim X + t$.

\begin{example}
    \rm{We work in $\mathbb{R}^n$ and suppose we are given some integer $1 \leq k \leq n-1$ and $0 < \beta \leq 1$. In this example, we construct a set $X \subset \mathbb{R}^n$ with dimension $\dim X = k - 1 + \beta$ and trapping dimension $T(X) = k + 1$. We then provide a lower bound for $\dim \mathcal{L}(X)$ from Corollary \ref{cor:fullbeck+Ren} and determine what happens as $\beta \rightarrow 1$.

    \medskip
    
    To this end, suppose that $S \subset \mathbb{R}^n$ is an embedding of the $k$-sphere $\mathbb{S}^k$ in $\mathbb{R}^n$ (which, of course, has dimension $k$). We then take a subset $X \subset S$ which satisfies $\dim X = k - 1 + \beta$. Now, we know that
    $$
    \dim (S \cap H) \leq k - 1, \quad \forall H \in \mathcal{A}(n,k) \text{ intersecting $S$}.
    $$
    Hence, for every $H \in \mathcal{A}(n,k)$, we have $\dim (X \setminus H) = \dim X$. So, the trapping number of $X$ satisfies $T(X) \geq k + 1$. However, the $k$-sphere $S$ is itself contained in some $k+1$ hyperplane, which guarantees that $T(X) \leq k + 1$. So, $X$ has trapping number $T(X) = k + 1$. We then have that
    $$
    T(X) - 1 = k \geq k - 1 + \beta = \dim X.
    $$
    This, in turn, implies that $\min \{2 \dim X, 2 (T(X) - 1) \} = 2 \dim X$, so that
    $$
    \dim \mathcal{L}(X) \geq \max\{\dim X+ t, 2\dim X\} =  2 \dim X.
    $$
    A similar result is, of course, given by Theorem \ref{thm:fullbeck}, since we may take $t = \dim X$. However, we note that
    $$
    \dim X = k - 1 + \beta \rightarrow 2 k \textrm{ as } \beta \rightarrow 1.
    $$
    In particular, whenever $X$ is a full-dimensional subset of $S$, we see that
    $$
    \dim \mathcal L(X) \geq 2 \dim X = 2 (T(X) - 1) = 2k.
    $$
    In the special case where $k = n - 1$, this inequality is an equality, since one necessarily has $\dim \mathcal{L}(X) \leq 2 (n - 1)$ for any $X \subset \mathbb{R}^n$.
    }
\end{example}

The last example demonstrates a set for which the lower bound of Corollary \ref{cor:fullbeck+Ren} is, in fact, an equality. However, we do not expect our bound to be sharp for all Borel sets $X$. Instead, we conjecture the following to be true.

\begin{conjecture}\label{conj:productlines}
Let $X \subset \mathbb{R}^n$ and let $T(X) = T$ be the trapping number of $X$. Suppose that $H \in \mathcal{A}(n, T)$ is the unique hyperplane such that
$
\dim (X \setminus H) < \dim X,
$
and set $X_1 = X \cap H$ and $X_2 = X \setminus H$. Then, one has
\begin{equation}\label{eq:conjecture}
\dim \mathcal{L}(X) = \max \{ \dim (X_1 \times X_2), \min \{\dim (X_1 \times X_1), 2(T - 1)\} \}
\end{equation}
\end{conjecture}

The spirit behind Conjecture \ref{conj:productlines} is the following. For $X = X_1 \sqcup X_2$ in $\mathbb{R}^n$ let $\mathcal{L}(X_1, X_2)$ denote the family of affine lines connecting pairs of points between $X_1$ and $X_2$. Then, we can write
$$
\mathcal{L}(X) : = \mathcal{L} (X_1) \cup \mathcal{L}(X_1, X_2) \cup \mathcal{L}(X_2).
$$
So, if one had the equalities
$$
\dim \mathcal{L}(X_1, X_2) = \dim (X_1 \times X_2), \quad \dim \mathcal{L}(X_1) = \dim (X_1 \times X_1),
$$
then this would prove Conjecture \ref{conj:productlines}. In particular, since:
$$
\dim (X_1 \times X_2) \geq \dim X_1 + \dim X_2 = \dim X + \inf_{H \in \mathcal{A}(n,T)} \dim (X \setminus H),
$$
$$
\dim (X_1 \times X_1) \geq 2 \dim X_1 = 2 \dim X,
$$
we see that Conjecture \ref{conj:productlines} in fact implies the lower bound of our main result, which is Theorem \ref{thm:fullbeck}.

\begin{remark}
    \rm{We note that \textit{bilinear} line set estimates of the form $\mathcal L(X,Y)$ can be studied in a similar manner, which is a direction we hope to pursue in future work.}
\end{remark}

To help motivate Conjecture \ref{conj:productlines} and demonstrate that it supersedes Theorem \ref{thm:fullbeck} for certain sets, we recall the following result concerning the dimension of product sets (see \cite[Theorem 8.10]{mattila1999} for a reference).

\begin{theorem}
    Let $A \subset \mathbb{R}^n$ and $B \subset \mathbb{R}^m$ be non-empty Borel sets. Then
    \begin{equation}\label{eq:prodboundmattila}
    \dim (A \times B) \geq \dim A + \dim B.
    \end{equation}
\end{theorem}
Many product sets $K = A \times B$ attain \eqref{eq:prodboundmattila} with equality--for example, when $A$ has equal Hausdorff and packing dimensions. Yet, the inequality \eqref{eq:prodboundmattila} can also be strict. Such product sets are precisely those which motivate Conjecture \ref{conj:productlines}. We give an example in $\mathbb{R}^2$ below, where point-line duality arguments allow us to easily calculate the dimension of $\mathcal{L}(X_i, X_j)$ in terms of the dimension of the product set $X_i \times X_j$ for $i,j = 1,2$.

We now construct a set $E \subset \mathbb{R}^2$ with $\dim (E) = 0$ such that $\dim \mathcal{L}(E) = 1$. Moreover, our set $E$ has the property that $E = E_1 \sqcup E_2$ where each $E_i$ is contained in some line $\ell_i$ and $\mathcal{L} (E) = \dim (E_1 \times E_2) = 1$. In particular, then, this example demonstrates a set $E$ which simultaneously sharpens the conjectured inequality \eqref{eq:conjecture} while also departing from the lower bound of Corollary \ref{cor:fullbeck+Ren}.

\begin{example}\label{ex:conjsharpness}
\rm{We adapt an example from \cite[p.90]{bishop_peres_2016} on $\mathbb{R}$ to lines in $\mathbb{R}^2$. For any $S \subset \mathbb{N}$, let
$$
A_S : = \{x \in [0,1] : \sum\limits_{k = 1}^{\infty}x_k 2^{-k}\}, \textrm{ where }
x_k \in
\begin{cases}
    \{0,1\}, \, k \in S \\
    \{0\}, \quad k \in S^C
\end{cases}.
$$
Note that $[0,1] \subset A_{S} + A_{S^C}$ for any choice of $S \subset \mathbb{N}$. It is shown in \cite{bishop_peres_2016} that
$$
\overline{\dim_{\mathscr{M}}} A_{S} = \lim_{N \rightarrow \infty} \frac{\# (S \cap \{1,...,N\})}{N},
$$
where $\overline{\dim_{\mathscr{M}}} A_S$ denotes the upper Minkowski dimension of $A_S$. Now, choose $S$ so that
$$
S : = \mathbb{N} \cap \bigg(\bigcup_{j = 1}^{\infty} [(2j - 1)!, 2j! )\bigg) \quad \quad S^C : = \mathbb{N} \cap \bigg(\bigcup_{j = 1}^{\infty} [2j!, (2j + 1)!)\bigg),
$$
and let $A = A_S$ and $B = A_{S^C}$.
We then have that
$$
\lim_{N \rightarrow \infty} \frac{\# (S \cap \{1,...,N\})}{N} = \lim_{N \rightarrow \infty} \frac{\# (S^C \cap \{1,...,N\})}{N} = 0,
$$
so that $\dim A = \dim B = 0$.

Now, let $E \subset \mathbb{R}^2$ be defined as $E = E_1 \sqcup E_2$ where
$$
E_1 : = \{(a,0) : a \in A \} \quad \quad \quad E_2 : = \{(0,b) : b \in B \}.
$$
Then $\dim (E) = 0$ and each $E_i$ is a subset of the two coordinate axes $\ell_x$ and $\ell_y$. We know, then, that
$$
\dim E = T(E) = \dim E + \inf_{H \in \mathcal{A}(2,k)} \dim (E \setminus H) = 0 \quad \quad \forall k = 0,1,2,
$$
so that Corollary \ref{cor:fullbeck+Ren} gives the trivial lower bound of $\dim \mathcal{L}(E) \geq 0$. 

Nevertheless, one has $\dim \mathcal{L}(E) = 1.$ To verify this, first write
$$
\mathcal{L}(E) = \mathcal{L}(E_1) \sqcup \mathcal{L}(E_2) \sqcup \mathcal{L}(E_1, E_2),
$$
and notice that $\dim \mathcal{L}(E_1) = \dim \mathcal{L}(E_2) = 0$. Now, since $\ell_x$ and $\ell_y$ are perpendicular, no lines in $\mathcal{L}(E_1, E_2)$ are vertical. So, we may use the point-line duality dimension bound (see, for example, \cite{orponen2023ABC} for a discussion) to obtain
$$
\dim \mathcal{L}(E_1, E_2) = \dim (E_1 \times E_2) = \dim (A \times B) = 1.
$$
In particular, this shows the \textit{equality}--not inequality--conjectured in \eqref{eq:conjecture}.
}
\end{example}

\begin{remark}
    \rm{Example \ref{ex:conjsharpness} demonstrates the improvement of Conjecture \ref{conj:productlines} over Corollary \ref{cor:fullbeck+Ren} by first finding two sets $A, B \subset \mathbb{R}$ such that $\dim (A \times B) > \dim A + \dim B$ and then embedding these sets in perpendicular lines in $\mathbb{R}^2$. It is likely that a similar method could give further examples relevant to Conjecture \ref{conj:productlines} in $\mathbb{R}^n$ for $n \geq 3$. That is, taking sets $A \subset \mathbb{R}^m$ and $B \subset \mathbb{R}^n$ such that $\dim (A \times B) > \dim A + \dim B$. Such sets are constructed in \cite{Hatano1971}, with parameters $\alpha = \dim A$, $\beta = \dim B$ and $\gamma = \dim (A \times B)$ satisfying $\alpha \leq m$ and
    $$
    \alpha + \beta < \gamma < \min \{ \alpha + m, \beta + n \}.
    $$
    We plan to study such generalizations at a later time.
    }
\end{remark}

\subsection{Outline of Paper} 

We begin by proving Theorem \ref{thm:fullbeck}, and thus Corollary \ref{cor:fullbeck+Ren}, in the plane in Section \ref{sec:plane}. This is done to outline the approach we will take in higher dimensions. We also take this time to develop useful lemmata regarding the affine Grassmannian and the relation to radial projections. Then, in Section \ref{sec:higherdim}, we discuss and prove the higher dimensional results. In this section, we generalize the planar lemmata and utilize results of Ren \cite{ren2023discretized} and the first author, Fu, and Ren \cite{brightfuren2024radial}.

\bigskip
\begin{sloppypar}
\noindent {\bf Acknowledgements.} A number of ideas for this paper were conceived while the authors were working on the Study Guide Writing Workshop 2023 at UPenn. We would like to thank our collaborators on the study guide: Ryan Bushling and Alex Ortiz. We would also like to thank Hong Wang for words of insight and encouragement when first attempting this problem at the Study Guide Workshop. Finally, we also must thank Josh Zahl for insightful discussions, especially for helping us formalize Corollary \ref{cor:fullbeck+Ren} and Conjecture \ref{conj:productlines}. 
\end{sloppypar}

\section{The Planar Theorem} \label{sec:plane}

\subsection{The Dimension of Line Families} \label{ss:planarprelims}
For each $n \geq 2$, we work with $\mathcal{G}(n,1)$, the set of linear subspaces of $\mathbb{R}^n$ and $\mathcal{A}(n,1)$, the family of affine lines in $\mathbb{R}^n$.

Before beginning our work on line families, we make a small note of the following notation. Whenever $x \in \mathbb{R}^n$, the map $\pi_x : \mathbb{R}^n \setminus \{x\} \rightarrow \mathbb{S}^{n-1}$ will denote radial projection onto the sphere centered at $x$; however, if whenever $L \in \mathcal{A}(n,1)$, we let $\pi_L : \mathbb{R}^n \rightarrow L$ denote the orthogonal projection map onto $L$.

We will view both $\mathcal{G}(n,1)$ and $\mathcal{A}(n,1)$ as a topological metric spaces. For $\mathcal{G}(n,1)$, we need only use the standard metric,
\begin{equation*}
    d_{\mathcal{G}(n,1)} (L, L') : = \lVert \pi_L - \pi_{L'}\rVert.
\end{equation*}
where $\pi_L$ denotes orthogonal projection onto $L$.

This metric induces a natural notion of \textit{measure} and \textit{dimension} on $\mathcal{G}(n,1)$, which is defined to the(normalized) $\mathcal{H}^s$ measures on $\mathbb{S}^{n-1}$ and radial projection about the origin. Since we often employ radial projection through the origin, we define it below.

\begin{definition}\label{eq:}
    For $x \in \mathbb{R}^n \setminus \{0\}$, we let $\pi_O (x)$ denote the radial projection of $x$ onto $\mathbb{S}^{n-1}$ centered at the origin. That is
    $$
    \pi_O (x) : = \frac{x}{|x|}, \quad \quad \forall x \in \mathbb{R}^n \setminus \{0\}.
    $$
\end{definition}
In particular, if $L \in \mathcal{G}(n,1)$ is any line through the origin, $\pi_O (L)$ consists of two antipodal points. Moreover, whenever $\mathcal{L} \subset \mathcal{G}(n,1)$ is (viewed as) a collection of lines in $\mathbb{R}^n$, we let
$$
\Theta_{\mathcal{L}} : = \pi_{O} (\bigcup_{L \in \mathcal{L}} L)  = \bigcup_{L \in \mathcal{L}} (L \cap \mathbb{S}^{n-1}),
$$
In words, $\Theta_{\mathcal{L}}$ is the set of directions determined by $\mathcal{L}$. We then use a notion of measure and dimension on $\mathcal{G}(n,1)$, which is discussed, for example, in \cite{mattila1999}.

\begin{definition}\label{def:hausmeasgrass}
For each $0 \leq s \leq n - 1$, the \textbf{s-Hausdorff measure} $\gamma_{s}$ on $\mathcal{G}(n,1)$ is defined as,
$$
\gamma_{s} (\mathcal{L}) : = \mathcal{H}^{s}\big\vert_{\mathbb{S}^{n-1}} \big( \Theta_{\mathcal{L}} \big), \quad \forall \, \mathcal{L} \subset \mathcal{G}(n,1),
$$
The \textbf{Hausdorff dimension} of $\mathcal{L}$ is then defined in the usual way but relative to the measures $\gamma_s$.
\end{definition}

For the affine Grassmanian $\mathcal{A}(n,1)$, we use a family of metrics parametrized by $w \in \mathbb{R}^n$.

\begin{definition}
For each $w \in \mathbb{R}^n$, we define a metric $d^w := d_{\mathcal{A}(n,1)}^w (\cdot, \cdot)$ on $\mathcal{A}(n,1)$ by setting,
\begin{equation*}
    d_{\mathcal{A}(n,1)}^w (\ell, \ell') = \lVert\pi_{L_w} - \pi_{L_{w}'}\rVert + |a_w - a_{w}'|, \quad \forall \ell, \ell' \in \mathcal{A}(n,1).
\end{equation*}
where $\pi_{L_w}$ denotes projection onto the line $L_w$, which is parallel to $\ell$ and passes through $w$, and $a_w$ is the point of intersection of the lines $\ell$ and the $n-1$-dimensional hyperplane through $w$ which is perpendicular to $L_w$.
\end{definition}

This family of metrics $d_{\mathcal{A}(n,1)}^w$ allows us to perform the following important ``change-of-coordinates'' for a family of lines passing through a common point. Indeed, notice that if $\mathcal{L}_w$ is a family of lines passing through $w \in \mathbb{R}^n$, then $a_w = a_{w}' = w$ for each $\ell, \ell' \in \mathcal{L}_{w}$. Consequently
$$
\sup_{\ell, \ell' \in \mathcal{L}_w} d^w (\ell, \ell') = \sup_{L_w, L_{w}'} \lVert\pi_{L_w} - \pi_{L_{w}'}\rVert = \sup_{L, L'} \lVert\pi_{L} - \pi_{L'}\rVert,
$$
where $L \sim L_w$ and $L' \sim L_{w}'$ are the linear subspaces parallel to $L_w$ and $L_{w}'$. So, relative to the metric $d^w$, the diameter of a collection of lines $\mathcal{L}_w$ passing through $w \in \mathbb{R}^n$ is computable in terms of $d_{\mathcal{G}(n,1)}$--the standard metric on the Grassmanian of linear subspaces. 

For us, the most important property of the metrics $d^w$s are their equivalence.

\begin{proposition}\label{thm:affinemetricequivalence}
    Suppose that $w_1$ and $w_2$ are points in $\mathbb{R}^n$. Then, the associated metrics $d^{w_1}$ and $d^{w_2}$ are bilipschitz equivalent. That is, there exist constants $C_1, C_2$--which are allowed to depend on $w_1$ and $w_2$--such that:
    \begin{equation}\label{eq:normequiv}
C_1 d^{w_1} (\ell, \ell^{'}) \leq d^{w_2} (\ell, \ell') \leq C_2 d^{w_1} (\ell, \ell')
    \end{equation}
    for all lines $\ell, \ell' \in \mathcal{A}(n,1)$.
\end{proposition}

\begin{proof}
    The ideas behind this Proposition are discussed in \cite{mattila1999}. However, we give a proof of the following restricted special case, as it is essentially the situation we encounter in the following key result (Lemma \ref{thm:linefamdim}). Specifically: we will be studying line families that intersect in a common point--so, the following special case considers the transformation from the origin to this point of intersection.

    \medskip
    Suppose $\ell, \ell' \in \mathcal{A}(n,1)$ are non-parallel lines, and set $w : = \ell \cap \ell'$. Let us show that there exist constants $C_1$ and $ C_2$ such that
    $$
    C_1 d^{w} (\ell, \ell') \leq d^O (\ell, \ell') \leq C_2 d^w (\ell, \ell')
    $$
    where $O$ is the origin. Importantly, our constants $C_1$ and $C_2$ will depend upon the distance $W : = |O - w| = |w|$, which is valid since our choice of $w$ fixes the metric, but not the lines in question.

\begin{figure}[h!]
    \centering
    \includegraphics[width=.5\textwidth]{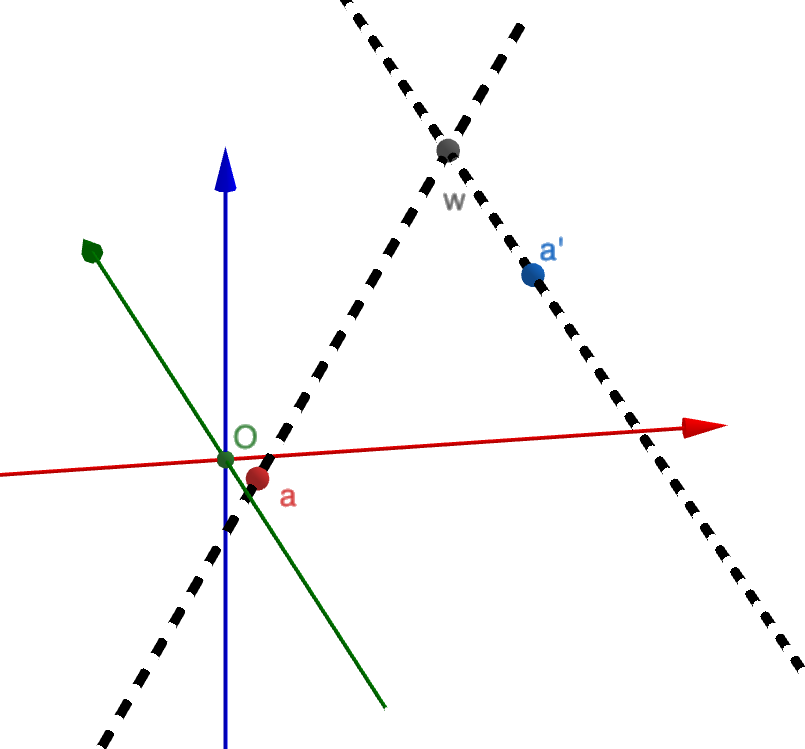}
    \caption{We work with the change-of-coordinates which takes us from the origin to the point of intersection of the lines $\ell$ and $\ell'$. This change-of-coordinates is highly important for what comes later.}
    \label{fig:specialcasemetric}
\end{figure}
    
    It is clear that the projection operators satisfy,
    $$
    \lVert\pi_{L_w} - \pi_{L_{w}'}\rVert = \lVert\pi_{L} - \pi_{L'}\rVert.
    $$
    which follows directly from the definition of the operator norm. So, call this value $\delta > 0$. Then, we have that
    $$
    d^w (\ell, \ell') = \delta
    $$
    $$
    d^O (\ell, \ell') = \delta + |a - a'|
    $$
    It is obvious that the left-hand inequality of \eqref{eq:normequiv} holds with constant $C_1 = 1$. It remains to show the right-hand side of the inequality is valid.

    Consider the two triangles $T$ and $T'$ formed by connecting the points $w$, $a$ and $O$ and $w$, $a'$ and $O$, respectively. Notice that, the assumption that $a \in L^{\perp}$ and $a' \in L'^{\perp}$ implies that both $T$ and $T^{'}$ are right-triangles, and both have hypotenuse along the segment connecting $w$ and $O$. Moreover, if $\theta$ and $\theta'$ denote the the angles of triangles $T$ and $T'$ (resp.) at vertex $w$, then we have
    $$
    \theta + \theta' \lesssim \delta.
    $$
    This is just another way of asserting that both $\theta$ and $\theta'$ are proportionally bounded above by the separation-in-direction of $\ell$ and $\ell'$.
    
\begin{figure}[h!]
\centering
\begin{subfigure}{.5\textwidth}
  \centering
  \includegraphics[width=.45\linewidth]{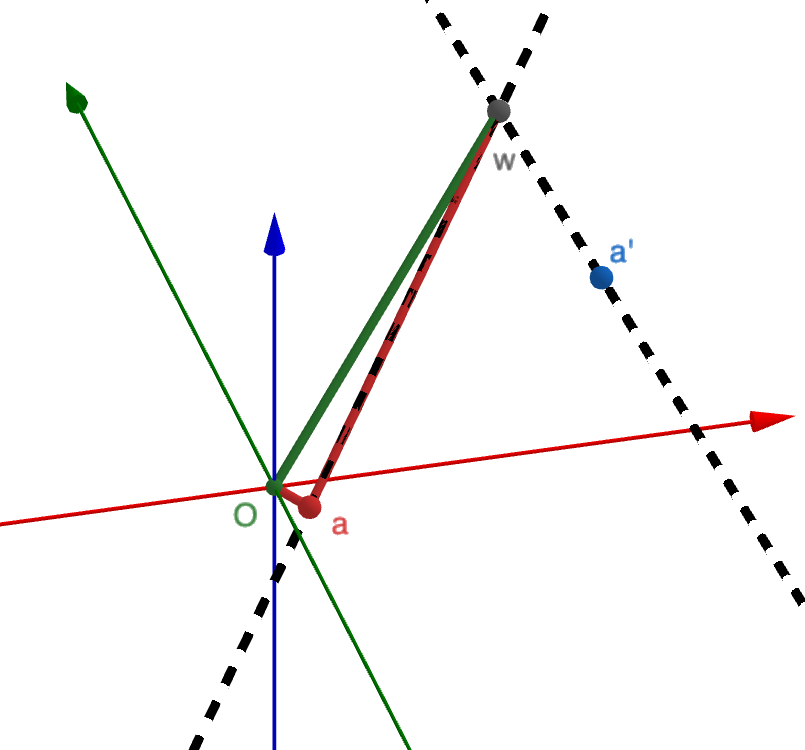}
  \caption{The right triangle formed by the origin, $a$ and $w$.}
  \label{fig:sub1}
\end{subfigure}%
\hspace{.2cm}
\begin{subfigure}{.45\textwidth}
  \centering
  \includegraphics[width=.5\linewidth]{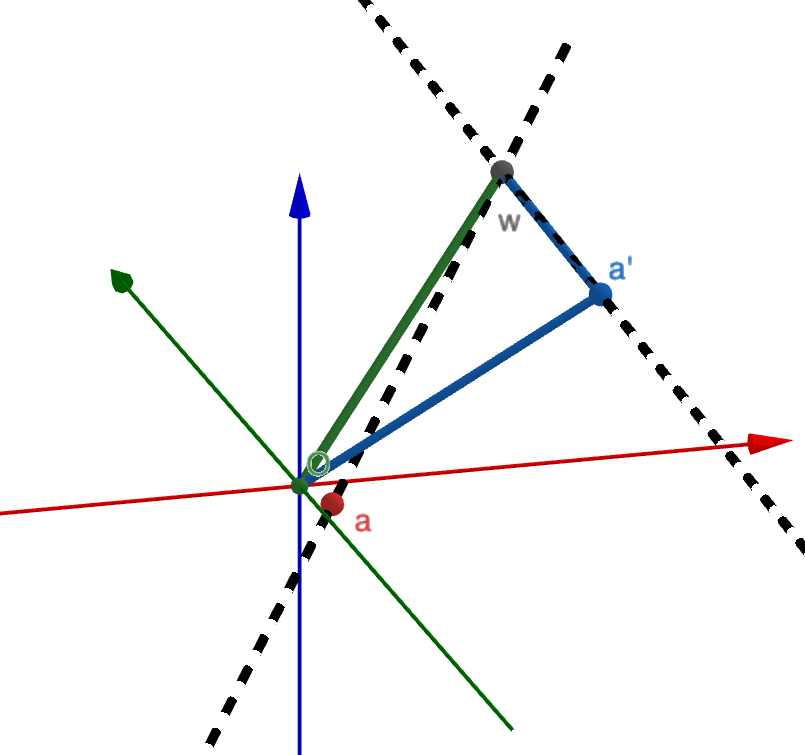}
  \caption{The right-triangle formed by the origin, $a'$ and $w$.}
  \label{fig:sub2}
\end{subfigure}
\caption{Both of these right-triangles are contained in two separate two-dimensional hyperplanes. Hence, we can apply elementary trigonometry and the triangle inequality to estimate their dimensions.}
\label{fig:test}
\end{figure}

\noindent

Combining our previous observations, we have
    $$
    |a - a'| \leq |a - O| + |a' - O| = \sin \theta |W| + \sin \theta' |W| \lesssim (\theta + \theta') |W|.
    $$
    Since $\theta + \theta' \lesssim \delta$, this gives:
    $$
    d^O (\ell, \ell') = \delta + |a - a'| \lesssim \delta |W| = |W| \cdot d^w (\ell, \ell').
    $$
    Hence, it suffices to take $C_2 \gtrsim |W|$. The proof is finished for this special case.
\end{proof}

Proposition \ref{thm:affinemetricequivalence} guarantees that any $s$-dimensional measures we generate from the metrics $d^w$ are guaranteed to be absolutely continuous. In particular, then, we can define our notion of $s$-dimensional Hausdorff measure on $\mathcal{A}(n,1)$ relative to the metric $d_{\mathcal{A}(n,1)}^O$ (so $w$ is chosen to be the origin). This is the metric given in \cite{mattila1999}. This brings us to the following Lemma--which we think of as an identity for calculating the dimension of line families which happen to have common intersection at a point.

\begin{lemma}\label{thm:linefamdim}
    Let $\mathcal{L}_{x}$ be a family of lines in $\mathbb{R}^n$, and suppose there exists a common point $x \in \ell$ for each $\ell \in \mathcal{L}_{x}$. Then,
    \begin{equation*}
        \dim \mathcal{L}_{x} = \dim \pi_{x}\bigg( \bigcup_{\ell \in \mathcal{L}} \ell \setminus \{x\} \bigg).
    \end{equation*}
\end{lemma}

\begin{proof}
    The proof is a summation of what we have already discussed. Namely, since our notion of dimension is stable under translation, we can freely choose $x = O$ to be the origin. This is critical, because in this special situation, $\mathcal{L}_{O}$ can be taken as a subset of $\mathcal{G}(n,1)$. In particular, we can define the dimension of our line family $\mathcal{L}_x$ relative to the measure $\gamma_s$ defined in Definition \ref{def:hausmeasgrass}. The details--again, freely assuming now that $x = O$ due to Proposition \ref{thm:affinemetricequivalence}--are as follows:
    \begin{eqnarray*}
    \dim \mathcal{L}_{x} & : = & \sup \{s : \gamma_s (\mathcal{L}_{O}) = 0 \} \\[1ex]
    \quad & = & \sup \{ s : \mathcal{H}^s (\Theta_{\mathcal{L}_{O}}) = 0 \} \\[1ex]
    \quad & = & \sup \big\{ s : \mathcal{H}^s \big(\bigcup_{\ell \in \mathcal{L}_{O}} (\ell \cap \mathbb{S}^{n-1})\big) = 0\big\} \\[1ex]
    \quad & = & \dim \pi_{O}\bigg( \bigcup_{\ell \in \mathcal{L}} (\ell \setminus \{O\}) \bigg)
    \end{eqnarray*}
\end{proof}

\begin{remark}
\rm{In the special case of $\mathcal{A}(2,1)$, one may leverage the duality of points and lines to greatly simplify the notion of dimension (at least in certain special cases). A discussion of this idea is given in \cite{orponen2023ABC} and \cite{orponen2022kaufman}. However, when working in $\mathcal{A}(n,1)$ for $n \geq 3$, this duality is not available to us. For this reason, we think of Lemma \ref{thm:linefamdim} as a substitute for this duality--at least in the special case where our line family has common intersection in a point.}
\end{remark}

These lemmata on dimension of line sets, as well as classic lower bounds on the dimension of Furstenberg sets, are utilized in the proof of Theorem \ref{thm:fullbeck}.

\subsection{Furstenberg Set Estimates} \label{sec:s1.1}

We now go into some background regarding Furstenberg set estimates, which has been a rapidly developing topic over the past few years and one that we will utilize in this paper. An $(s,t)$-\textit{Furstenberg set} in $\R^n$ is a set of points $F$ such that there exists a (non-empty) family of lines $\mathcal L\subset \mathcal A(n,1)$ such that $\dim \mathcal L\geq t>0$, and such that $\dim (F\cap \ell) \geq s>0$ for all $\ell \in L.$ 

The Furstenberg set problem strives to find a lower bound on the dimension of all $(s,t)$-Furstenberg sets--a problem that was recently solved by Kevin Ren and Hong Wang in the plane \cite{ren2023furstenberg}. Before this result was proven, there were a number of partial results towards the full conjecture which prove useful for this paper. One such result is the the following:
\[
\dim F \geq s + \min\{s,t\}
\]
where $F$ is an $(s,t)$-Furstenberg set. The case when $t\leq s$ was proven by Lutz and Stull using information theory \cite{lutz2017bounding} and the case $s\leq t$ is essentially due to Wolff \cite{wolff1999recent}. Both cases were also proven by H\'era--Shmerkin--Yavicoli in [\cite{héra2020improved}, Theorem A.1]. 

However, studying (generalized) dual Furstenberg sets has been proving fruitful in the past few years. A (generalized) \textit{dual $(s,t)$-Furstenberg set of lines} in $\R^n$ is a set of lines $\mathcal L$ such that there exists a non-empty set of pins $F^\ast \subset \R^n$ such that $\dim F^\ast \geq t>0$, and for all $x\in F^\ast$, there exists a set of lines $\mathcal L_x\subset \mathcal L$ through $x$ such that $\dim \mathcal L_x \geq s \geq 0$. The (generalized) dual Furstenberg set problem then asks for a lower bound on $\dim \mathcal L$. As it happens, in two dimensions, the notion of a dual Furstenberg set is precisely dual to a Furstenberg set using point-line duality (see \cite{orponen2022kaufman} Section 3.1 for more details). In particular, via point-line duality in two dimensions, one obtains 
\[
\dim \mathcal L \geq s + \min \{s,t\}
\]
where $\mathcal L$ is a dual $(s,t)$-Furstenberg set.
This lower bound is the one we shall utilize in the proof of the continuum Erd\H{o}s--Beck Theorem.

\begin{remark}
    \rm{While in higher dimensions, generalized dual Furstenberg sets are \textit{not} dual to a Furstenberg set, the notion still generalizes the lower dimensional case, and are thusly referred to as generalized dual Furstenberg sets. See Section \ref{ss:higherdimprelims} for an analogous lower bound.} 
\end{remark}

\subsection{Planar Erd\H{o}s--Beck Theorems for Fractal Sets} \label{ss:planarstuff}

We now prove our main result in the plane, which we restate here for convenience.

\begin{theorem}\label{thm:planebeck}
    Let $X\subset \R^2$ be Borel. Then, 
    \begin{enumerate}
    \item[\textbf{1)}] if $\dim (X\setminus \ell) = \dim X$ for every $\ell \in \mathcal A(2,1)$, then $\dim \mathcal L(X) \geq \min\{2\dim X, 2\}$.
    \item[\textbf{2)}] if there exists an $\ell \in \mathcal A(2,1)$ such that $\dim (X\setminus \ell) < \dim X$, we let $0 < t \leq \dim X$ such that $\dim (X\setminus \ell) \geq t$ for all $\ell \in \mathcal A(2,1)$. Then,
        \[
        \dim \mathcal L(X) \geq \dim X + t.
        \]
    \end{enumerate}
\end{theorem}

\begin{proof}
If it is the case that for all lines $\ell$ we have $\dim (X\setminus \ell) = \dim X$, then we may apply the continuum version of Beck's theorem and obtain
\[
\dim \mathcal L(X) \geq \min \{2\dim X, 2\}.
\]

Otherwise, there exists an $\ell \in \mathcal A(2,1)$ such that $\dim X\setminus \ell < \dim X$. Fix $0 < t \leq \dim X$ such that $\dim (X\setminus \ell') \geq t$ for all lines $\ell'$. Given there exists a line $\ell$ such that $\dim (X\setminus \ell) < \dim X$, we have that $\dim (X\cap \ell) = \dim X := s$. Therefore, by Lemma \ref{thm:kplanelemma}, we have that 
    \[
    \dim \pi_x(X\setminus \{x\}) \geq s, \quad \forall x\in X\setminus \ell.
    \]
    Hence, by Lemma \ref{thm:linefamdim}, the set of lines through $x$, $\mathcal L_x\subset \mathcal L(X)$, satisfies 
    \[
    \dim \mathcal L_x = \dim \pi_x(X\setminus \{x\}) \geq s
    \]
    for all $x\in X\setminus \ell$.

    Hence, we have that
    \[
    \mathcal L(X) \supset \bigcup_{x\in X\setminus \ell} \mathcal L_x,
    \]
    where $\dim (X\setminus \ell) \geq t$ (by assumption) and $\dim \mathcal L_x \geq s$ for all $x\in X\setminus \ell$. This is an dual $(s,t)$-Furstenberg set (note that here we used $X\setminus \ell$ as our non-empty set of pins with dimension at least $t>0$ by assumption). Thus, using point-line duality, we see that $$\dim \left(\bigcup_{x\in X\setminus \ell} \mathcal L_x\right) \geq \dim F,$$ where $F$ is an $(s,t)$-Furstenberg set. Note that by assumption, $0 < t \leq s = \dim X$. Thus, the result of Lutz--Stull gives that $\dim F \geq s + t$. In total, we have that 
    \[
    \dim \mathcal L(X) \geq \dim \left(\bigcup_{x\in X\setminus \ell} \mathcal L_x\right) \geq \dim F \geq s + t.
    \]
This gives the desired result.   
\end{proof}

\begin{corollary}
    Let $X\subset \R^2$ be Borel. Then, 
    \begin{enumerate}
        \item[\textbf{1)}] if $\dim (X\setminus \ell) = \dim X$ for every $\ell \in \mathcal A(2,1)$, then $\dim \mathcal L(X) \geq \min\{2\dim X, 2\}$.
    \item[\textbf{2)}] if there exists an $\ell \in \mathcal A(2,1)$ such that $\dim (X\setminus \ell) < \dim X$, we let $0 < t \leq \dim X$ such that $\dim (X\setminus \ell) \geq t$ for all $\ell \in \mathcal A(2,1)$. Then,
        \[
        \dim \mathcal L(X) \geq \max\{\dim X + t, \min\{2\dim X, 2(T(X)-1)\}.
        \]
    \end{enumerate}
\end{corollary}
\begin{proof}
    Compare Theorem \ref{thm:planebeck} with Theorem \ref{thm:rensBeck} (due to Ren) to obtain the result.
\end{proof}

\section{The Higher Dimensional Theorem} \label{sec:higherdim}

\subsection{Preliminaries for Higher Dimensions}\label{ss:higherdimprelims}

In the course of proving Theorem \ref{thm:fullbeck}, we will use the following two lemmata.

\begin{lemma}\label{thm:kplanelemma}
    Let $X \subset \mathbb{R}^n$ with $\dim X = s \in [k, k+1)$ for some integer $1 \leq k \leq n - 1$. Assume that there exists a $k$-plane $H \in \mathcal{A}(n,k)$ so that $\dim (X \cap H) = s$. Then,
    \begin{equation}\label{eq:kplanelemma}
        \dim \pi_x (X \setminus \{x\}) \geq s, \quad \forall x \in X \setminus H.
    \end{equation}
\end{lemma}

\begin{proof}[Proof of Lemma \ref{thm:kplanelemma}]
    Firstly, note that if $X\setminus H =\emptyset$, this is vacously true, so we may assume that $X\setminus H \neq \emptyset.$

    Call $X_1 := X \cap H$ and $X_2 := X \setminus H$. We show that,
    \begin{equation*}
        \dim \pi_{x_2} (X_1) = s, \quad \forall x_2 \in X_2.
    \end{equation*}
    which clearly implies \eqref{eq:kplanelemma}. Indeed, this follows because the mapping $\pi_{x_2} : H \rightarrow \mathbb{S}^{n-1}$ is locally bi-Lipschitz for each fixed $x_2 \in X_2$. However, since locally bi-Lipschitz mappings preserve dimension, we have
    \begin{equation*}
        \dim \pi_{x_2} (X_1) = \dim X_1 = s.
    \end{equation*}
    Hence, to finish the proof, we show that the mapping $\pi_{x_2} : X_1 \rightarrow \mathbb{S}^{n-1}$ is locally bi-Lipschitz.
    
    We will assume (without loss of generality) that
    $$
    H : = \{x \in \mathbb{R}^n : x_{k+1} = \cdots = x_{n-1} = 0; x_n = 1 \}
    $$ 
    and that $x_2 = 0$. Note that, for an arbitrary $k$-dimensional hyperplane $H \subset \mathbb{R}^n$ and point $x_2 \in \mathbb{R}^n$, we can always perform an affine transformation to arrive at this special case. Since affine transformations preserve Hausdorff dimension, we are fine making this reduction. In this special case, notice that 
    $$\pi_{x_2} (H) = \{\theta \in \mathbb{S}^{n-1} : \theta_{k+1} = \cdots = \theta_{n-1} =0 \} = \Theta.
    $$ 
    So, the mapping $\pi_{x_2}|_{H}$ admits an inverse function $\psi : \Theta \rightarrow H$, defined via:
    \begin{equation*}
        \psi (\theta_1,...,\theta_k, \theta_{k+1},..., \theta_n) : = \left(\frac{\theta_1}{\theta_n},...,\frac{\theta_k}{\theta_n}, 0,...,0, 1\right).
    \end{equation*}

    Now, for each $n = 1,2,3,...$, let
    $$
    X_n : = \{x \in X_1 : |x - e_n| \leq n \},
    $$
    where $e_n$ is the $n$-th standard normal vector in $\mathbb{R}^n$. Similarly, let $\Theta_n : = \pi_{x_2} (X_n)$, which is a subset of $\Theta$. The mapping $\pi_{x_2}|_{X_n}$ is Lipschitz with constant $\sim 1$. Moreover, this mapping admits an inverse $\psi_n : \Theta_n \rightarrow X_n$, which is also Lipschitz with constant $\sim n$. Therefore,
    $$
    \dim \pi_{x_2} (X_n) := \dim \Theta_n = \dim X_n.
    $$
    In particular,
    \begin{equation*}
    \dim \pi_{x_2} (X_1) \geq \sup_{1 \leq n < \infty} (\dim X_n) = \dim X.
    \end{equation*}
    Since $X_1 \subset X$, this concludes the proof.
\end{proof}

In addition to the previous two lemmata, we will require some recent results--- the first of which is a continuum higher dimensional Beck's theorem by Ren, and the second of which is a dual Furstenberg set estimate from the first author, Fu, and Ren---both of which we state without proof.
\begin{theorem}[\cite{ren2023discretized}, Corollary 1.5] \label{thm:higherdimBeck}
    Let $X\subset \R^n$ be a Borel set with $\dim(X\setminus H) = \dim X$ for all $k$-planes $H$. Then, the line set $\mathcal L(X)$ spanned by pairs of distinct points in $X$ satisfies 
    \[
    \dim \mathcal L(X) \geq\min\{2\dim X, 2k\}.
    \]
\end{theorem}

\begin{theorem}[\cite{brightfuren2024radial}, Theorem 4] \label{thm:higherdimFurstenberg} 
Let $0\leq s \leq n-1$ and let $0<t\leq n$.
Given $\mathcal L$ is a dual $(s,t)$-Furstenberg set of lines in $\R^n$, we have, 
\[
\dim \mathcal L \geq s + \min \{s,t\}.
\]
\end{theorem}

\subsection{Higher Dimensional Erd\H{o}s-Beck Theorems for Fractal Sets} \label{ss:higherdimtheorem}

We now prove the higher-dimensional Erd\H{o}s-Beck Theorem, Theorem \ref{thm:fullbeck}, restated here.

\begin{theorem}\label{thm:higherdimtheorem}
    Let $X\subset \R^n$ be Borel, and let $k\in [1,n-1]$ be an integer. Then, 
    \begin{enumerate}
        \item[\textbf{1)}] if $\dim (X\setminus H) = \dim X$ for every $H \in \mathcal A(n,k)$, then $\dim \mathcal L(X) \geq \min\{2\dim X, 2k\}$.
        \item[\textbf{2)}] if there exists an $H \in \mathcal A(n,k)$ such that $\dim (X\setminus H) < \dim X$, we let $0 < t \leq \dim X$ such that $\dim (X\setminus P) \geq t$ for all $P \in \mathcal A(n,k)$. Then,
    \end{enumerate}
    \[
    \dim \mathcal L(X) \geq \dim X + t.
    \]
\end{theorem}

\begin{proof}
    Fix an integer $k\in [1,n-1]$. If it is the case that for all $k$-planes $H \in \mathcal A(n,k)$, we have $\dim (X\setminus H) = \dim X$, then by Theorem \ref{thm:higherdimBeck} from \cite{ren2023discretized}, we have that 
    \[
    \dim \mathcal L(X) \geq \min \{2\dim X, 2k\}.
    \]

    Otherwise, there exists an $H$ such that $\dim (X\setminus H) < \dim X$. Then, fix $0< t \leq \dim X$ such that $\dim (X\setminus P) \geq t$ for all $k$-planes $P\in \mathcal A(n,k)$. We now go into case work. Given there exists a $k$-plane $H\in \mathcal A(n,k)$ such that $\dim (X\setminus H) < \dim X$, we have that $\dim (X\cap H) = \dim X := s$. Therefore, by Lemma \ref{thm:kplanelemma}, we have that 
    \[
    \dim \pi_x(X\setminus \{x\}) \geq s, \quad \forall x \in X \setminus H.
    \] Hence, by Lemma \ref{thm:linefamdim}, for all $x\in X\setminus H$, the set of lines through $x$, $\mathcal L_x\subset \mathcal L(X)$, satisfies
    \[
    \dim \mathcal L_x = \dim \pi_x\left(X\setminus \{x\}\right) \geq s
    \]
    for all $x\in X\setminus H$.

    In total, we have that
    \[
    \mathcal L(X) \supset \bigcup_{x\in X\setminus H} \mathcal L_x,
    \]
    where $X\setminus H$ is non-empty, $\dim X\setminus H \geq t$, and $\dim \mathcal L_x \geq s$ for all $x\in X\setminus H$. By Theorem \ref{thm:higherdimFurstenberg}, this implies that 
    \[
    \dim \mathcal L(X) \geq \dim \left(\bigcup_{x\in X\setminus H} \mathcal L_x\right) \geq s+ \min \{t,s\} = s + t := \dim X + t.
    \]
\end{proof}

\begin{corollary}
    Let $X\subset \R^n$ be Borel and let $k\in [1,n-1]$ be an integer. Then, 
    \begin{enumerate}
        \item[\textbf{1)}] if $\dim (X\setminus H) = \dim X$ for every $H \in \mathcal A(n,k)$, then $\dim \mathcal L(X) \geq \min\{2\dim X, 2k\}$.
        \item[\textbf{2)}] if there exists an $H \in \mathcal A(n,k)$ such that $\dim (X\setminus H) < \dim X$, we let $0 < t \leq \dim X$ such that $\dim (X\setminus P) \geq t$ for all $P \in \mathcal A(n,k)$. Then,
        \[
        \dim \mathcal L(X) \geq \max\{\dim X + t, \min\{2\dim X,2(T(X)-1)\}\}.
        \]
    \end{enumerate}
\end{corollary}

\begin{proof}
    Compare Theorem \ref{thm:higherdimtheorem} with Theorem \ref{thm:rensBeck} (due to Ren) to obtain the result.
\end{proof}

\printbibliography

\end{document}